\documentclass[11pt]{amsart}
\usepackage[sorted-cites]{amsrefs}
\usepackage{verbatim}
\usepackage{color}
\usepackage{amsmath,amsthm,amscd}
\usepackage{calrsfs}
\DeclareMathAlphabet{\pazocal}{OMS}{zplm}{m}{n}
\usepackage{lipsum}

\renewcommand{\bar}{\overline}

\newtheorem{theorem}{Theorem}[section]

\newtheorem{lemma}{Lemma}[section]
\newtheorem{cor}{Corollary}[section]
\newtheorem{prop}{Proposition}[section]
\newtheorem{claim}{Claim}

\theoremstyle{remark}
\newtheorem{remark}{Remark}[section]
\numberwithin{equation}{section}

\renewcommand{\bar}{\overline}

\begin{document}

\title[Volume growth of submanifolds]{Volume growth of complete submanifolds in gradient Ricci Solitons with bounded weighted mean curvature}

\date{}

 \subjclass[2000]{Primary: 58J50;
Secondary: 58E30}

\thanks{The first and third authors are partially supported by CNPq and Faperj of Brazil.}

\address{Instituto de Matem\'atica e Estat\'\i stica, Universidade Federal Fluminense,
Niter\'oi, RJ 24020, Brazil}
\email{xucheng@id.uff.br}

\author[Xu Cheng]{Xu Cheng}

\address{Departamento de Matem\'atica, Universidade Federal do Esp\'irito Santo,
Vit\'oria, 29075-910, Brazil}
\email{matheus.vieira@ufes.br}
\author[Matheus Vieira]{Matheus Vieira}

\address{Instituto de Matem\'atica e Estat\'\i stica, Universidade Federal Fluminense,
Niter\'oi, RJ 24020, Brazil}
\email{zhou@impa.br}
\author[Detang Zhou]{Detang Zhou}

\newcommand{\M}{\mathcal M}

\begin{abstract} 
In this article, we study properly immersed complete noncompact submanifolds   in a  complete shrinking gradient Ricci soliton with weighted mean curvature vector bounded in norm.  We prove that   such a submanifold must have  polynomial volume growth under some mild assumption on the potential function. On the other hand, if the ambient manifold is of bounded geometry, we prove that such a submanifold must have at least linear volume growth.   In particular, we show that   a properly immersed complete noncompact  hypersurface   in the Euclidean space with bounded Gaussian weighted mean curvature must  have  polynomial volume growth   and  at least linear volume growth.

\end{abstract}

\maketitle
\section{introduction}\label{introduction}

In recent years, motivated by the research of the mean curvature flows  in the Euclidean space $\mathbb{R}^m$,   self-shrinkers, self-expanders and translators have been studied extensively as they are the  singularity  models for the flows.  It is known that they are the critical points of the corresponding  weighted volume functionals for all compactly supported variations, which indicates that it is important to study   submanifolds under  weighted volume measures.

Recently, McGonagle-Ross \cite{MR} began to  study  a special class of hypersurfaces in $\mathbb{R}^m$: 
the  solutions to the Gaussian isoperimetric problem, which satisfy the equation 
\begin{align}\label{lambda}
H-\dfrac{\left<x, \bf{n}\right>}{2}=\lambda,
\end{align}
where $\lambda\in \mathbb{R}$ is constant.  These hypersurfaces are not only constant weighted mean curvature (CWMC) hypersurfaces in $\mathbb{R}^m$, but also  the critical points of the Gaussian  weighted area  functional  for compactly supported variations preserving  enclosed Gaussian weighted volume. A trivial example is any hyperplane as $H=0$ and $\left<x, \bf{n}\right>$ is constant. In \cite{MR}, they proved that  hyperplanes are the only stable smooth, complete, properly immersed solutions to the Gaussian isoperimetric problem, and that there are no hypersurfaces of index one. Later, in \cite{CW}, Q.M.Cheng-G.Wei called hypersurfaces satisfying  (\ref{lambda}) $\lambda$-hypersurfaces.  Also, they  extended the concept of $F$-functional of self-shrinkers introduced by Colding-Minicozzi in \cite{CM}  to  $\lambda$-hypersurfaces and studied the related $F$-stability. There are some other works, for instance, by Q.M.Cheng-Ogata-G.Wei \cite{COW},  Guang \cite{G}, and etc.
It is worth noting  that if $\lambda=0$ in  (\ref{lambda}), the hypersurface is just a self-shrinker.

The Gaussian isoperimetric problem in $\mathbb{R}^m$ can be generalized to a general ambient manifold.  In general,  a   constant weighted mean curvature (CWMC) hypersurface (see its definition in Section \ref{preliminary})  is  a critical point of the weighted area functional for compactly supported variations preserving enclosed weighted volume.   In this paper, the ambient smooth metric measure space $(M, \bar{g}, f)$ we consider   is a shrinking gradient Ricci soliton, that is,    the triple $\left(M,\overline{g},f\right)$
 satisfies that 
\begin{align}\overline{Ric}+\overline{\nabla}^2f=\frac12 \overline{g}, \label{1-def-soliton}
\end{align}
where, for convenience,  we choose the constant $\frac12$ and   the potential function $f$ to be normalized (see the meaning of the normalization of $f$ in Section \ref{section-upper}).  It is known that shrinking gradient Ricci solitons are  very important in research of the Ricci flow since they are  singularity models  of type I of the Ricci flow. 
 Our  consideration on shrinking Ricci soliton ambient manifolds  not only includes the $\lambda$-hypersurfaces in  Gaussian isoperimetric problem in $\mathbb{R}^m$, but also    arises from the study of mean curvature flows of  hypersurfces in an  ambient manifold evolving by  Ricci flow.  In this aspect, 
Lott \cite{Lo} and Magni-Mantegazza-Tsatis \cite{MMT}    showed that Huisken's monotonicity formula holds  when the ambient is a   gradient Ricci soliton solution to the Ricci flow.  Later, Yamamoto \cite{Y} studied the  asymptotic behavior of a Ricci-mean curvature flow moving along a gradient shrinking Ricci soliton when it develops a singularity of type I. 
 Moreover, Lott  \cite{Lo} introduced the  concept of mean curvature soliton  for the mean curvature flow evolving  in a    gradient Ricci soliton solution. Its definition implies that  a mean curvature soliton is just an  $f$-minimal hypersurface in a gradient Ricci soliton,  which   is  the critical point of weighted area functional with the weight $e^{-f}$. Here   $f$ is the potential function of the ambient gradient Ricci soliton.  There are some studies on the properties of geometry and topology of  $f$-minimal hypersurfaces in a gradient shrinking soliton, or more general, an ambient manifold with Bakry-\'Emery Ricci curvature bounded below by a positive constant.   In   \cite{CMZ1} and  \cite{CMZ2}, Mejia, the first and third authors of the present paper discussed the stability of $f$-minimal surfaces and proved  some compactness theorems for  $f$-minimal surfaces. There are also other works, for instance, see \cite{CMZ3},   \cite{CZ13-2}, 
\cite{IR},   \cite{L}, \cite{VZ} and etc.

The class of $f$-minimal hypersurfaces is the particular case of CWMC hypersurfaces with  weighted mean curvature zero.  In this article,  motivated  by studying  the  volume growth of CWMC hypersurfaces in
shrinking gradient Ricci solitons, we  deal with   properly immersed submanifolds in shrinking gradient  Ricci solitons with weighted mean curvature vector  bounded in norm.  Indeed, the shrinking gradient Ricci soliton ambience leads to strong restrictions on the volume growth of such  submanifolds. It is known that the volume of a properly immersed minimal hypersurface in $\mathbb{R}^m$ has at least Euclidean growth and may grow exponentially. But, Ding-Xin \cite{DX} proved that a complete properly immersed self-shrinker hypersurface in $\mathbb{R}^{m}$ has polynomial volume growth. Further, in \cite{CZ1},  the first and third authors of the present paper, proved that for a complete self-shrinker  in $\mathbb{R}^{m}$, properness of  immersion, polynomial volume growth and finiteness of weighted volume are equivalent to each other.  Recently,  in the arxiv version of  \cite{CW},  Q.M.Cheng-G.Wei used  Theorem 1.1 in \cite{CZ1} to  the $\lambda$-hypersurfaces in $\mathbb{R}^{m}$ satisfying  the equation $ H-\left<x, \bf{n}\right>=\lambda$    and proved that such a $\lambda$-hypersurface also has a polynomial volume growth. On the other hand,  in  \cite{CMZ2}, Mejia,  the first  and  third authors of the present paper  studied  $f$-minimal submanifolds in a  shrinking  gradient Ricci soliton and showed the equivalence of  the properness of  immersion, polynomial volume growth and finiteness of weighted volume of a $f$-minimal submanifold under the assumption  of the convexity of $f$.   In this paper, we prove that  for  a complete properly immersed  submanifold in a shrinking gradient  Ricci soliton, if its weighted mean curvature vector is bounded in  norm, then it has polynomial volume growth (see Theorem \ref{th-4}).  Unfortunately, Theorem 1.1 in \cite{CZ1} cannot be applied directly to prove this result when the ambient manifold is a shrinking gradient  Ricci soliton satisfying (\ref{1-def-soliton}), even including  the case  of   CWMC hypersurfaces  in $\mathbb{R}^{m}$ satisfying  (\ref{lambda}).
In order to prove it, we need the following Theorem \ref{th-1}  which is the key in the proof of Theorem \ref{th-4} and  of independent interest.     
\begin{theorem}\label{th-1}
Let $\left(X,g\right)$ be a    complete noncompact Riemannian manifold.  Let 
$f$ be a proper  nonnegative  $C^2$  function. Assume that there exist constants  $\alpha>0$,  $\beta>0$, and $a$ so that
 $f$ satisfies
\begin{equation}
\Delta f-\alpha\left|\nabla f\right|^{2}+\beta f\le a\label{th1-eq-1}
\end{equation}
and
\begin{equation}
\Delta f\le a.\label{th1-eq-2}
\end{equation}
Then

\begin{itemize}
\item [(i)]    The volume of the set $D_{r}=\left\{x\in X;  2\sqrt{ f}\leq \sqrt{2\beta}r\right\} $, denoted by $\text{Vol}(D_r)$,
satisfies
\begin{equation}
\text{Vol}(D_{r})\leq  e^{\frac{\alpha\beta}{2}}r^{\frac{2\alpha a}{\beta}}\int_{D_{r}}e^{-\alpha f}dv, \quad \text{for} \quad r\geq 1.\label{th1-eq-3}
\end{equation}

\item [(ii)] The integral 
\begin{equation}
\int_{X}e^{-\alpha f}dv<\infty.\label{th1-eq-4}
\end{equation}

\item [(iii)] The volume of the  set  $D_r$ has polynomial growth. More precisely,  for $r\geq 1$,
\begin{equation}
\text{Vol}\left(D_{r}\right)\le Cr^{\frac{2\alpha a}{\beta}}, \quad \text{where} \quad C=e^{\frac{\alpha\beta}{2}}\int_{X}e^{-\alpha f}dv. \label{th1-eq-5}
\end{equation}
\end{itemize}
\end{theorem}
\begin{remark}Theorem \ref{th-1} generalizes     Theorem 1.1  in \cite{CZ1} proved by the first and third authors of the present paper. 
\end{remark}

Next, based on Theorem \ref{th-1} and the  identities of submanifolds in a shrinking gradient Ricci soliton, we  show
\begin{theorem}\label{th-4}
Let $\left(M,\bar{g},f\right)$ be an $m$-dimensional
complete shrinking gradient  Ricci soliton satisfying (\ref{1-def-soliton}).  Let $\Sigma$ be an $n$-dimensional, $n<m$,  properly immersed  complete
submanifold in $M$ with  the trace of Hessian of $f$ restricted on the normal bundle of $\Sigma$ satisfying $tr_{{\Sigma}^{\perp}}\bar{\nabla}^2f\geq\frac{k}{2}$
for some constant $k$. If $\Sigma$ has a weighted
mean curvature vector $\mathbf{H}_f$ bounded in norm, then

(i) The volume of the set $D_{r}:=\left\{x\in \Sigma; 2\sqrt{f}\leq r\right\} $
satisfies 
\begin{align}\label{intro-ineq-1}
\text{Vol}\left(D_{r}\right)\leq Cr^l,  \quad \text{for} \quad r\geq 1,
\end{align}
where    $l=m-k-\inf_{\Sigma}(\bar{R}+\left|\mathbf{H}\right|^{2})-\inf_{\Sigma}(\bar{R}+|(\bar{\nabla}f)^{\perp}|^{2})+\sup_{\Sigma}\left|\mathbf{H}_{f}\right|^{2}$ is a nonnegative constant, and $C=e^{\frac18+\frac12\inf_{\Sigma}(\bar{R}+|(\bar{\nabla}f)^{\perp}|^{2})}\int_{\Sigma}e^{-\frac{1}{2}f}.$

(ii) $\Sigma$ has finite weighted volume with respect to the weighted volume element  $e^{-f}dv_{\Sigma}$:
\begin{align}\label{intro-ineq-2}
\int_{\Sigma}e^{-f}<\infty.
\end{align}

(iii)  $\Sigma$ must have polynomial  volume growth. More precisely, for a fixed point  $p\in M$, there exist  constants $C$ and $r_0$ so  that for all $r\geq r_0$,
\[
\text{Vol}\left(B^M_{r}\left(p\right)\cap\Sigma\right)\le Cr^{l},
\]
where $l$ is the same constant as in (i),  $B_r^M(p)=\{x\in M; d_M(p,x)\leq r\}$ denotes the geodesic ball in $M$ of radius $r$ centered at $p$, and $ \text{Vol} ({B}^M_r( p)\cap \Sigma)$ denotes the volume of ${B}^M_r( p)\cap \Sigma$. 
\end{theorem}

It is worth noting that the  polynomial volume growth rate $r^l$ in Theorem \ref{th-4} is optimal. Observe that   the cylinder shrinkers $S^{d}(\sqrt{2d})\times \mathbb{R}^{n-d}$  in $\mathbb{R}^{n+1}$ and cylinder Gaussian CWMC hypersurfaces $S^{d}(R)\times \mathbb{R}^{n-d}$ in $\mathbb{R}^{n+1}$, $0\leq d< n,$  have the volume growth rate $r^{n-d}$ and in both cases, $l=n-d$.     (see Remark \ref{example-rate} for details).

In \cite{AR}, Alencar-Rocha  applied Theorem 1.1 in \cite{CZ1}  and proved the following result (\cite{AR}, Proposition 1.2):  Let $\left(M,\bar{g},f\right)$ be an $m$-dimensional
complete shrinking gradient  Ricci soliton satisfying (\ref{1-def-soliton}) with  convex  potential function $f$, i.e., $\bar{\nabla}^2f\geq 0$. If $\Sigma$ is a   properly immersed  complete noncompact 
submanifold in $M$  satisfying $b=\sup_{\Sigma}\left<\mathbf{H}_f, \bar{\nabla}f\right><\infty$, then $\Sigma$ has finite weighted volume and its volume has polynomial growth with the rate $r^{m+2b}$.   We mention that  for a CWMC hypersurface in $\mathbb{R}^m$ with nonzero constant  $H_f$,  it is not clear whether the assumption $\sup_{\Sigma}\left<\mathbf{H}_f, \bar{\nabla}f\right><\infty$ is satisfied  in general.

Using Proposition 1.1 in \cite{AR}, Proposition 5 in  \cite{CMZ2} and  Theorem \ref{th-4} in this paper, we prove the equivalence of properness of immersion, polynomial volume growth and  finite weighted volume for submanifols in a shrinking gradient soliton with weighted mean curvature bounded in norm. More precisely,
\begin{theorem}\label{th-equiv}
Let $\left(M,\bar{g},f\right)$ be an $m$-dimensional
complete shrinking gradient  Ricci soliton.  Let $\Sigma$ be an $n$-dimensional, $n<m$,    complete
submanifold immersed in $M$ with  the trace of Hessian of $f$ restricted on the normal bundle of $\Sigma$ satisfying $tr_{{\Sigma}^{\perp}}\bar{\nabla}^2f\geq \frac{k}{2}$
for some constant $k$. If  the weighted mean curvature vector $\mathbf{H}_f$ of $\Sigma$  is  bounded in  norm, then the following items are equivalent to each other:

(i) $\Sigma$ is properly immersed.

(ii) $\Sigma$ has polynomial volume  growth.

 (iii) $\Sigma$ has finite weighted volume:
\begin{align*}
\int_{\Sigma}e^{-f}<\infty.
\end{align*}
\end{theorem}
\begin{remark} Theorem \ref{th-equiv} generalizes Theorem 1.3 in \cite{CZ1} for the case of self-shrinkers and Corollary 1 in \cite{CMZ2} for the case of $f$-minimal submanifolds. 
\end{remark}

As a special case, we get
\begin{theorem}\label{th-2}
  Let  $f(x)=\frac{|x|^2}{4}$, $x\in \mathbb{R}^{n+1}$ be the Gaussian potential function.
 Let $\Sigma$ be a  complete
hypersurface immersed in $\mathbb{R}^{n+1}$ with  bounded  Gaussian weighted 
mean curvature $H_f=H-\frac{\left<x, \bf{n}\right>}{2}$. Then  the following items are equivalent to each other:

(i) $\Sigma$ is properly immersed.

(ii)  $\Sigma$  has at most  the following polynomial volume growth rate: For a fixed point $p\in M$,  there are some constants $C>0$ and $r_0$ so that for all $r\geq r_0$,
\[
\text{Vol}\left(B_{r}\left(p\right)\cap\Sigma\right)\le Cr^{l}, 
\]
where   $l=n-1-\inf_{\Sigma}\left|{H}\right|^{2}-\frac14\inf_{\Sigma}\left|\left<x, \bf{n}\right>\right|^{2}+\sup_{\Sigma}\left|{H}_f\right|^2$ and  $B_r(p)$ denotes the geodesic ball in $\mathbb{R}^{n+1}$ of radius $r$ centered at $p$.

(iii) $\Sigma$ has   polynomial  volume growth.

 (iv) $\Sigma$ has finite weighted volume:
\begin{align*}
\int_{\Sigma}e^{-\frac{|x|^2}{4}}<\infty.
\end{align*}

\end{theorem}

In Section \ref{section-lower} of this paper, we study the lower bound of volume growth of a complete noncompact submanifold properly immersed  in a  shrinking gradient  Ricci soliton with weighted mean curvature vector bounded in norm and  prove that it  has at least linear growth.  Here we give a brief related history. Calabi and Yau proved independently that a complete noncompact Riemannian manifold with nonnegative Ricci curvature must have at least  linear volume growth.  Recently, Munteanu-Wang   showed that  this lower bound estimate also holds for a shrinking gradient Ricci soliton in \cite{MW1}  and later generalized to the complete smooth metric measure space $(\Sigma, g, e^{-f}dv)$  with $Ric_f\geq \frac12g$ and $|\nabla f|^2\leq f$  in \cite{MW2 }.  On the other hand,  for  submanifolds,  Cheung-Leung \cite{CL}  proved that the volume of any complete noncompact submanifold in $\mathbb{R}^m$, $\mathbb{H}^m$ or more generally, in an ambient manifold of bounded geometry, must have at least linear growth.   Recently,  Li -Y. Wei \cite{LW} proved that a properly immersed complete non-compact self-shrinker in $\mathbb{R}^{n+1}$ also has at least linear volume growth. Later, the result of Li-Wei  was extended  to the $\lambda$-hypersurfaces in $\mathbb{R}^{n+1}$ by Q.M.Cheng-G.Wei in \cite{CW}, and to   the case of  $f$-minimal submanifolds in a shrinking gradient Ricci soliton with convex potential function $f$ by Y. Wei  in \cite{W}.  

Recall that a Riemannian manifold is  said to have  bounded geometry if  the  sectional curvature is bounded above and the  injectivity radius is bounded below by a positive constant. In this paper, we prove
\begin{theorem}\label{th-3}Let $\left(M,\bar{g},f\right)$ be an $m$-dimensional
complete  shrinking   gradient Ricci soliton satisfying (\ref{1-def-soliton}). Assume that $M$ has bounded geometry. Let $\Sigma$ be an $n$-dimensional complete
properly immersed submanifold in $M$ with weighted
mean curvature vector bounded in norm.   Then  $\Sigma$ must have at least  linear  volume growth, that is,  for a fixed point $p\in M$, there are some  constants  $C$  and $r_0$ so that
\begin{equation} \text{Vol} ({B}^M_r( p)\cap \Sigma)\geq Cr, \quad \text{ for all}\quad  r\geq r_0, \label{P-1-2}
\end{equation}
where $B_r^M(p)=\{x\in M; d_M(p,x)\leq r\}$ denotes the geodesic ball in $M$ of radius $r$ centered at $p$.
\end{theorem}
It is easy to show that the  volume growth rates of submanifolds we consider  are independent of the choice of the point $p\in M$.
Theorem \ref{th-3} has the following special case:
\begin{theorem}\label{th-cor-th3} Let  $f=\frac{|x|^2}{4}$, $x\in \mathbb{R}^{n+1}$ be the Gaussian potential function.
 Let $\Sigma$ be a  complete
hypersurface properly immersed in $\mathbb{R}^{n+1}$ with  bounded  Gaussian weighted 
mean curvature $H_f=H-\dfrac{\left<x, \bf{n}\right>}{2}$. Then $\Sigma$ must have at least  linear  volume growth.
 \end{theorem}

The rest of the paper is organized as follows: In Section \ref{preliminary} we give some notation and facts as preliminaries. In Section \ref{section-th-1} we prove Theorem \ref{th-1}. In Section \ref{section-upper} we prove Theorem \ref{th-4}.  In Section \ref{section-lower} we prove Theorem \ref{th-3}.

\section{Preliminaries}\label{preliminary}

In this section, we give some notation and conventions.

Let  $(X^d, g)$  denote a $d$-dimensional smooth Riemannian manifold. A smooth metric measure space, denoted by $(X^d, g, e^{-f}dv)$, is  $(X^d, g)$ together with a weighted volume form $e^{-f}dv$ on $X$, where $f$ is a smooth function on $X$ and $dv$ is the volume element induced by the metric $g$.  For $(X^d, g, e^{-f}dv)$, the  drifted Laplacian $\Delta_{f}=\Delta-\left\langle \nabla f,\nabla\cdot\right\rangle$ 
 is a densely defined self-adjoint second order elliptic  operator in $L^2(X, e^{-f}dv)$. For 
$u$ and $v$ in $C^{\infty}_0(X)$, it holds that 
\begin{equation}
\int_{X}(\Delta_{f}u) ve^{-f}dv=-\int_{X}\left\langle \nabla u,\nabla v\right\rangle e^{-f}dv.
\end{equation}

In this article, we will study submanifolds in a Riemannian manifold.   
Let  $\left(M,\bar{g}\right)$ be a smooth  $m$-dimensional  Riemannian manifold and $f$ a smooth function on $M$.  
Let $i: (\Sigma^n,g) \to (M^{m},\bar{g})$, $n<m$, denote the smooth  isometric  immersion of an $n$-dimensional submanifold $\Sigma$ into $M$.    Clearly the function $f$ restricted on $\Sigma$, still denoted by $f$,  induces a weighted volume  element  $e^{-f}dv_{\Sigma}$ on $\Sigma$ and thus   a smooth metric measure space $(\Sigma, g, e^{-f}dv_{\Sigma})$, where   $dv_{\Sigma}$ denote the volume element of $(\Sigma, g)$.
When we deal with  submanifolds, unless otherwise specified, the notations with a bar denote the quantities
corresponding  to the metric $\bar{g}$.  On the other hand, the notations  without a bar  denote the quantities corresponding to  the intrinsic metric ${g}$ on $\Sigma$.

The  submanifold $\Sigma$ is said to be properly immersed if,  for any compact subset $\Omega$ in $M$, the pre-image $i^{-1}(\Omega)$ is compact in $\Sigma$.

For $\Sigma$, its  mean curvature vector  ${\bf H}$   is 
\[{\bf H}(p):=\sum_{i=1}^n(\bar\nabla_{e_i}e_i)^\perp=\sum_{i=1}^nA(e_i,e_i),\quad  p\in \Sigma,
\]
where $\{ e_1, e_2, \cdots, e_n\}$ is a local orthonormal frame of  $\Sigma$ at $p$,    $\perp$ denotes the projection onto the normal bundle of $\Sigma$, and $A$ denotes   the second fundamental form of  $\Sigma$.

 The weighted mean curvature vector ${\bf H}_f$  of  $\Sigma$  is defined  by 
\begin{equation}{\bf H}_f:={\bf H}+(\overline{\nabla }f)^{\perp}.
\end{equation}

A submanifold  $\Sigma$  is called $f$-minimal  if  its weighted mean curvature ${\bf H}_f$ vanishes identically, or equivalently ${\bf H}=-(\overline{\nabla} f)^\perp.$ 

The weighted volume of  a measurable subset  $S\subset \Sigma$ is  defined by  
$$V_f(S)=\int_S e^{-f}dv_{\Sigma}.$$
It is known that  an $f$-minimal submanifold  is a critical point of the weighted volume  functional.  
On the other hand,  it is also  a minimal submanifold under the conformal metric $\tilde{g}=e^{-\frac{2}{n}f}\bar{g}$ on $M$  (see, e.g. \cite{CMZ2}, \cite{CMZ1}).

In the case of hypersurfaces, the mean curvature of $\Sigma$  is defined by 
\begin{equation}{\bf H}=-H{\bf n},
\end{equation}
where ${\bf n}$ is the unit normal field on $\Sigma$.  

The weighted mean curvature is defined as
\begin{equation}{\bf H}_f=-H_f{\bf n}, \quad \text{or equivalently}\quad H_f=H-\left<\bar{\nabla}f, {\bf n}\right>
\end{equation}
A  hypersurface $\Sigma$ is said to be a  constant weighted mean curvature (denoted simply by CWMC) hypersurface if it satisfies 
\begin{align}\label{def-CWMC}H-\left<\bar{\nabla}f, {\bf n}\right>=C.
\end{align}

So  an $f$-minimal hypersurface  $\Sigma$ is   the case of $C=0$,  that is,
$H=\left<\bar{\nabla}f, {\bf n}\right>$.
When $(M, \bar{g})$ is the Euclidean space $(\mathbb{R}^{m}, g_0)$ with the Gaussian measure $e^{-\frac{|x|^2}{4}}dv_{g_0}$, a CWMC hypersurface is just the solution to the Gaussian isoperimetric problem in \cite{MR} or the $\lambda$-hypersurface in \cite{CW}  (observe that in \cite{CW}, the weighted measure is  taken to be $e^{-\frac{|x|^2}{2}}dv_{g_0}$, which has no essential difference).

In \cite{BC} and \cite{BCE}, Barbosa-do Carmo-Eschenburg showed that a constant mean curvature (CMC)  hypersurface in a Riemannian manifold  is the critical point of  its area functional for compactly supported variations which preserve enclosed volume. Mcgonagle-Ross \cite{MR} proved the same property  for CWMC hypersurfaces  in $\mathbb{R}^m$ under  the  Gaussian  weighted area and Gaussian weighted volume, respectively. Using the similar argument to the one in \cite{BCE} and \cite{MR}, one  can show that a CWMC hypersurface in a Riemannian manifold   is  still  the critical point of  its weighted area functional for compactly supported variations which preserve enclosed weighted volume, where  the weighted volume element in $M$ is $e^{-f}dv_M$ and $dv_M$ denotes the volume element of $M$.

In this paper, we estimate the volume growth of submanifolds. Recall that $\Sigma$ is said to have  polynomial volume growth  if, for a fixed point $p\in M$, there are constants $C$,  $s$ and $r_0$ so that for all $r\geq r_0$,
\begin{equation} \text{Vol} ({B}^M_r( p)\cap \Sigma)\leq Cr^s, \label{P-1-1}
\end{equation}
where ${B}^M_r( p)$ is the geodesic  ball in $M$ of radius $r$ centered at $p$,  $ \text{Vol} ({B}^M_r( p)$ denotes the volume of ${B}^M_r( p)\cap \Sigma$.
When $s=n$ in (\ref{P-1-1}), $\Sigma$ is said to be of  Euclidean volume growth. 



In this article, the ambient manifold $M$ is assumed to be a shrinking  gradient Ricci soliton. The triple $\left(M,\overline{g},f\right)$
is called a shrinking gradient  Ricci soliton if it satisfies that 
\begin{align}\overline{Ric}+\overline{\nabla}^2f=\rho \overline{g}, \label{def-soliton-rho}
\end{align}
where $\rho>0$ is a constant. Gaussian shrinking soliton $(\mathbb{R}^{n+1}, g_{can}, \frac{|x|^2}{4})$ and Cylinder shrinking solitons are examples.  If  $f$ is constant, a shrinking  gradient Ricci soliton is just an Einstein manifold.

\section{Proof of Theorem \ref{th-1}}\label{section-th-1}

Recall that a function  on a Riemannian manifold  is called proper if  for any compact set in $\mathbb{R}$, its  pre-image is compact. Now we prove Theorem \ref{th-1}. 
\bigskip

\noindent{\it Proof of Theorem \ref{th-1} }.   Since $f$ is proper and nonnegative,  the minimum of $f$ can be achieved in some point of the manifold $X$. Thus (\ref{th1-eq-1}) implies that the constant $a\geq 0$.
Denote $\gamma=\frac{\beta}{\alpha}$.
 For $t\geq1$, define the function
\[
\phi_{t}=a\log t+\frac{\alpha}{t^{\gamma}}f
\]
and consider the drifted Laplacian $\Delta_{\phi_t}=\Delta-\left\langle \nabla\phi_{t},\nabla \cdot\right\rangle$. 

Using  Assumptions (\ref{th1-eq-1}) and  (\ref{th1-eq-2}), we have  
\begin{align}
\left(\Delta_{\phi_{t}}f\right)e^{-\phi_{t}}+t\frac{d}{dt}e^{-\phi_{t}}
 =&\left( \Delta f-\frac{\alpha}{t^{\gamma}}\left|\nabla f\right|^{2}-a+\frac{\beta}{t^{\gamma}}f\right)e^{-\phi_{t}}\nonumber\\
 =&\left[ \frac{1}{t^{\gamma}}(\Delta f-\alpha\left|\nabla f\right|^{2}+\beta f-a)\right] e^{-\phi_{t}}\nonumber\\
 \quad& +\left[ (1-\frac{1}{t^{\gamma}})(\Delta f-a)\right] e^{-\phi_{t}}\nonumber\\
 \leq & 0.\label{delta_phi_ineq}
\end{align}
Then, (\ref{delta_phi_ineq}) implies that,  at a regular value of $f$ and for $t\geq 1$,
\begin{align}\label{3-ineq-8}
\frac{d}{dt}\int_{D_{r}}e^{-\phi_{t}} =&\int_{D_{r}}\frac{d}{dt}e^{-\phi_{t}}\nonumber\\
  \leq&-\frac{1}{t}\int_{D_{r}}\left(\Delta_{\phi_{t}}f\right)e^{-\phi_{t}}=-\frac{1}{t}\int_{\partial D_{r}}\langle\nabla f,\frac{\nabla f}{\left|\nabla f\right|}\rangle e^{-\phi_{t}}\nonumber\\
  =&-\frac{1}{t}\int_{\partial D_{r}}\left|\nabla f\right|e^{-\phi_{t}} \leq 0, 
\end{align}
where we used  the Stokes' theorem   in the second equality. Integrating (\ref{3-ineq-8}) from $t=1$ to $t=T\geq 1$ gives
\[
\int_{D_{r}}e^{-\phi_{T}}\leq\int_{D_{r}}e^{-\phi_{1}}.
\]
Taking $T=r^{\frac{2}{\gamma}}, r\geq 1$ and noting $\phi_{1}=\alpha f$, we have
\begin{equation}\label{eq-3-5}
r^{-\frac{2a}{\gamma}}\int_{D_{r}}e^{-\frac{\alpha }{r^{2}}f}\leq\int_{D_{r}}e^{-\alpha f}.
\end{equation}
 Since the integrals in (\ref{eq-3-5}) is right continuous in $r$,  (\ref{eq-3-5}) holds for all $r\geq 1$.  By $D_{r}=\left\{x\in X;  2\sqrt{ f}\leq \sqrt{2\beta}r\right\} $ and  $\gamma=\frac{\beta}{\alpha}$,  (\ref{eq-3-5}) implies  
\begin{equation}\label{eq-3-6}
\text{Vol}\left(D_{r}\right)\leq e^{\frac{\alpha\beta}{2}} r^{\frac{2\alpha a}{\beta}}\int_{D_{r}}e^{-\alpha f}, \quad \text{for} \quad r\geq 1,
\end{equation}
which is   (\ref{th1-eq-3}). Now we prove (ii).
We have that, for $r\geq 1$,
\begin{align}\label{3-ineq-9}
\int_{D_{r}}e^{-\alpha f}-\int_{D_{r-1}}e^{-\alpha f} & =\int_{D_{r}\backslash D_{r-1}}e^{-\alpha f} \leq e^{-\frac{\alpha\beta\left(r-1\right)^{2}}{2}}\int_{D_{r}\backslash D_{r-1}}1\nonumber\\
 & \leq e^{-\frac{\alpha\beta\left(r-1\right)^{2}}{2}}\text{Vol}\left(D_{r}\right)\nonumber\\
 &\leq e^{\frac{\alpha\beta}{2}}r^{\frac{2\alpha a}{\beta}}e^{-\frac{\alpha\beta\left(r-1\right)^{2}}{2}}\int_{D_{r}}e^{-\alpha f}.
\end{align}
In  (\ref{3-ineq-9}), the  last  inequality  used Inequality (\ref{eq-3-6}).
Noting that there is a large  number $r_0\geq 1$ such that  for $r\geq r_0$,  $e^{\frac{\alpha\beta}{2}}r^{\frac{2\alpha a}{\beta}}e^{-\frac{\alpha\beta\left(r-1\right)^{2}}{2}}\leq e^{-r}$, we have, by (\ref{3-ineq-9}),
\[
\int_{D_{r}}e^{-\alpha f}\leq\frac{1}{1-e^{-r}}\int_{D_{r-1}}e^{-\alpha f}.
\]
Thus, given any positive integer $N$ and $r\geq r_0$,
\begin{align}\label{3-ineq-10}
\int_{D_{r+N}}e^{-\alpha f}\leq\left(\prod_{i=0}^{N}\frac{1}{1-e^{-(r+i)}}\right)\int_{D_{r-1}}e^{-\alpha f}.
\end{align}
 Noting that  the infinite product $\displaystyle \prod_{i=0}^{\infty}\left(1-e^{-(r+i)}\right)$ converges to a positive number and letting $N$ tend to infinty in (\ref{3-ineq-10}),  we get 
\begin{align}\label{3-ineq-11}
\int_{X}e^{-\alpha f}<\infty,
\end{align}
which is (ii). Finally   (\ref{eq-3-6})
and  (\ref{3-ineq-11}) give  (iii):
\begin{align}
\text{Vol}\left(D_{r}\right)\leq Cr^{\frac{2a}{\gamma}}=Cr^{\frac{2\alpha a}{\beta}}, \quad \text{for}\quad  r\geq 1, \nonumber
\end{align}
where $C=e^{\frac{\alpha\beta}{2}}\int_{X}e^{-\alpha f}$. 

\qed

\begin{remark}  Theorem \ref{th-1} generalizes Theorem 1.1 in \cite{CZ1} by the following fact:  since $\Delta f-|\nabla f|^2+f\leq a$ and $|\nabla f|^2\leq f$ imply the inequality $\Delta f\leq a$,  we can take $\alpha=1$, $ \beta=1 $ in Theorem \ref{th-1}.
\end{remark}

\section{Upper estimate of volume growth}\label{section-upper}

In what follows,  we assume that  the ambient space $\left(M^m,\overline{g},f\right)$
is an $m$-dimensional smooth  shrinking gradient Ricci soliton. Without loss of generality, we take $\rho=\frac12$ in (\ref{def-soliton-rho}),  that is,   $\left(M,\overline{g},f\right)$ satisfies
\begin{align}\label{4-eq-soliton-1}
\overline{Ric}+\overline{\nabla}^2f=\frac12 \overline{g}.
\end{align}
Equation  (\ref{4-eq-soliton-1})  implies that 
\begin{align}
\overline{R}+\overline{\Delta}f&=\frac{m}{2}.\label{4-eq-soliton-2}
\end{align}
Here $\bar{R}$ denotes the scalar curvature of $M$. It is well known that the potential function $f$  satisfying  (\ref{4-eq-soliton-1}) can be normalized by adding a suitable constant to it  such that 
the following identity holds:
\begin{align}
\bar{R}+\left|\bar{\nabla}f\right|^{2}&= f. \label{4-eq-soliton-4}
\end{align}
Then (\ref{4-eq-soliton-2}) and (\ref{4-eq-soliton-4}) give that
 \begin{align}
\bar{\Delta}f-\left|\bar{\nabla}f\right|^{2}+ f&=\frac{m}{2}.\label{4-eq-soliton-3}
\end{align} 

It was proved by B. Chen \cite{C} that the scalar curvature of $M$ satisfies  $\bar{R}\geq 0$. Hence
\begin{align}\label{4-eq-soliton-5} |\overline{\nabla}f|^2\leq f.
\end{align}
By applying Theorem \ref{th-1}, we may prove Theorem \ref{th-4}.

\bigskip

\noindent {\it Proof of Theorem \ref{th-4}}. 
In \cite{CZ},  Cao and the third author proved the following result: For a fixed point  $p\in M$,  let  $\bar{r}(x):=d_M(p,x)$, $x\in M$, denote the distance between $x$ and $p$ in $M$. Then there are positive constants $c$ and $r_0$ depending only of   $f(p)$ and dimension $m$ such that for any $x\in M$ and $\bar{r}(x)\geq r_0$,
\begin{align}\label{f-estimate}\frac{1}{4}(\bar{r}(x)-c)^2\leq f(x)\leq \frac{1}{4}(\bar{r}(x)+c)^2.
\end{align}
Inequality  (\ref{f-estimate}) implies that $f$ is proper on $M$.  This property of $f$ and  the proper immersion of   $\Sigma$  in $M$ imply that the restriction $ f|_{\Sigma}$  of $f$ on $\Sigma$ is also proper. 
Identity (\ref{4-eq-soliton-4}) implies that
\begin{align}
\left|\nabla f\right|^{2} & =\left|\bar{\nabla}f\right|^{2}-|(\bar{\nabla}f)^{\perp}|^{2}= f-\bar{R}-|(\bar{\nabla}f)^{\perp}|^{2}. \label{gradient-f}
\end{align}
Using  (\ref{4-eq-soliton-2}) and $\mathbf{H}_{f}=\mathbf{H}+(\bar{\nabla}f)^{\perp}$, we have
\begin{align}\label{4-ineq-Delta}
\Delta f & =tr_{\Sigma}\bar{\nabla}^2f+\left\langle \bar{\nabla}f,\mathbf{H}\right\rangle\nonumber \\
 & =\bar{\Delta}f-tr_{{\Sigma}^{\perp}}\bar{\nabla}^2f+\langle (\bar{\nabla}f)^{\perp},\mathbf{H}_{f}\rangle -|\left(\bar{\nabla}f\right)^{\perp}|^{2}\nonumber\\
& =\frac12m-\bar{R}-tr_{\Sigma^{\perp}}\bar{\nabla}^2f-\frac{1}{2}\left|\mathbf{H}\right|^{2}+\frac{1}{2}\left|\mathbf{H}_{f}\right|^{2}-\frac12|(\bar{\nabla}f)^{\perp}|^{2},
\end{align}
where we used the basic identity: $\langle u,v\rangle=\frac12(|u|^2+|v|^2-|u-v|^2)$ in the last equality of (\ref{4-ineq-Delta}).
Thus, by   $\bar{R}\geq 0$,   the boundedness of $\mathbf{H}_f$, and the assumption $tr_{\Sigma^{\perp}}\bar{\nabla}^2f\geq \frac{k}{2}$, we have
\begin{align} \label{4-ineq-Delta-1}
\Delta f\leq\frac{l}{2},
\end{align}
where
$l=m-k-\inf_{\Sigma}(\bar{R}+\left|\mathbf{H}\right|^{2})-\inf_{\Sigma}(\bar{R}+|(\bar{\nabla}f)^{\perp}|^{2})+\sup_{\Sigma}\left|\mathbf{H}_{f}\right|^{2}.$

Since $f=f|_{\Sigma}$ is a  proper nonnegative  function on $\Sigma$,  the minimum of $f$ is achieved on $\Sigma$. So  $\Delta f\geq 0$ at the minimal point and hence $l\geq 0$. On the other hand, using (\ref{gradient-f}) and (\ref{4-ineq-Delta}), we have 
\begin{align}
\Delta f-\frac{1}{2}\left|\nabla f\right|^{2}+\frac12 f
& =\frac12 m-\bar{R}-tr_{\Sigma^{\perp}}\bar{\nabla}^2f-\frac{1}{2}\left|\mathbf{H}\right|^{2}+\frac{1}{2}\left|\mathbf{H}_{f}\right|^{2}-\frac12|(\bar{\nabla}f)^{\perp}|^{2}\nonumber\\
 &\qquad -\dfrac12( f-\bar{R}-|(\bar{\nabla}f)^{\perp}|^{2})+\frac12 f\nonumber\\
&=\frac12 m-\frac{1}{2}\bar{R}-tr_{\Sigma^{\perp}}\bar{\nabla}^2f-\frac{1}{2}\left|\mathbf{H}\right|^{2}+\frac{1}{2}\left|\mathbf{H}_{f}\right|^{2}.\label{4-ineq-Delta-Grad}
\end{align}
Letting $\tilde{f}=f-\inf_{\Sigma}(\bar{R}+|(\bar{\nabla}f)^{\perp}|^{2})$, by (\ref{gradient-f}), we have  $\tilde{f}\geq 0$. By (\ref{4-ineq-Delta-1}), 
\begin{align}
\Delta\tilde{ f}\leq\frac{l}{2}.
\end{align}
Also, (\ref{4-ineq-Delta-Grad}) implies 
\begin{align}
&\Delta \tilde{f}-\frac{1}{2}\left|\nabla \tilde{f}\right|^{2}+\frac12\tilde{ f}\leq \frac l2.
\end{align}

Now  in Theorem \ref{th-1} we let  $X=\Sigma$, $f=\tilde{f}$, $\alpha=\frac12, \beta=\frac12, $ and $ a=\frac{l}{2} $.  By the conclusion (ii) in Theorem \ref{th-1},  we have that 
$$e^{\frac12\inf_{\Sigma}(\bar{R}+|(\bar{\nabla}f)^{\perp}|^{2})}\int_{\Sigma}e^{-\frac{1}{2}f}=\int_{\Sigma}e^{-\frac{\tilde{f}}{2}}<\infty.$$
The conclusion (iii) in Theorem \ref{th-1} implies that 
the volume of the  set $\{ x\in \Sigma; 2\sqrt{ \tilde{f}}\leq r\} $ satisfies
\begin{align}\label{level-vol}
\text{Vol}(\{ x\in \Sigma; 2\sqrt{ \tilde{f}}\leq r\})\leq Cr^l, \quad \text{for} \quad r\geq 1,
\end{align}
where the constant $C=e^{\frac18+\frac12\inf_{\Sigma}(\bar{R}+|(\bar{\nabla}f)^{\perp}|^{2})}\int_{\Sigma}e^{-\frac{1}{2}f}$.
Since the level set $D_{r}=\left\{ x\in \Sigma; 2\sqrt{ {f}}\leq r\right\}\subseteq \{ x\in \Sigma; 2\sqrt{ \tilde{f}}\leq r\} $,  its volume satisfies that
$$\text{Vol}(D_r)\leq Cr^l, \quad \text{for} \quad r\geq 1,$$
where the constant $C$ is the same as in (\ref{level-vol}). Thus we proved (i) of Theorem \ref{th-4}.
Noting that  $\int_{\Sigma}e^{-f}\leq \int_{\Sigma}e^{-\frac{f}{2}}, $ we get (ii) of Theorem \ref{th-4}.
Finally, we prove (iii).  Inequality (\ref{f-estimate}) implies that for $r\geq \max\{r_0,1\}$,
\begin{align*}
B^M_r(p)\cap\Sigma\subseteq \{x\in \Sigma:f\leq \frac{1}{4}( r+c)^2\}=D_{r+c}.
\end{align*}
So, for $r\geq \max\{r_0,1, c\}$,
\begin{align}\label{4-ineq-12}\text{Vol}(B^M_r(p)\cap\Sigma)\leq \text{Vol}(D_{r+c})\leq C(r+c)^l\leq C_1r^l,
\end{align}
where the constant $C_1=2^lC=2^le^{\frac18+\frac12\inf_{\Sigma}(\bar{R}+|(\bar{\nabla}f)^{\perp}|^{2})}\int_{\Sigma}e^{-\frac{1}{2}f}$. This completes the proof of   Theorem  \ref{th-4}.

\qed

\begin{remark}\label{example-rate}
Theorem \ref{th-4} extends the known results in \cite{CMZ2} and  \cite{DX}. Moreover,
the order $l$ of the polynomial  volume growth estimate  is optimal.  In $f$-minimal case, i.e.,  $\mathbf{H}_f=\mathbf{0}$, we may take the cylinder self-shrinkers  $S^{d}(\sqrt{2d})\times \mathbb{R}^{n-d}$ in $\mathbb{R}^{m}$,   $0\leq d< n$.  Then  $|\mathbf{H}|=\frac{\sqrt{d}}{\sqrt{2}}$ and  $l=n-d$. Theorem \ref{th-4} implies that $ \text{Vol}( B_{r}(0)\cap \Sigma)\leq C r^{n-d}.$ Obviously the order $n-d$ is achieved. In  the case of $\mathbf{H}_f\neq \mathbf{0}$,  we may take any cylinder CWMC hypersurface $S^{d}(R)\times \mathbb{R}^{n-d}$ in $\mathbb{R}^{n+1}$, where $0\leq d< n$. Then $H=\frac{d}{R},$ $ \left<x,\mathbf{n}\right>=R,$ and $H_f=H-\frac12\left<x,\mathbf{n}\right>=\frac{d}{R}-\frac{R}{2}$. Thus $l=n-d$, which is also  sharp.
\end{remark}
\begin{remark}\label{rho-arbi} Theorem \ref{th-1} can be applied for all the values of the constant $\rho$ when   $M$ satisfies the general shrinking gradient Ricci soliton equation (\ref{def-soliton-rho}):  $\overline{Ric}+\overline{\nabla}^2f=\rho \overline{g}$ with the constant $\rho>0$ and the normalized $f$, that is, $\bar{R}+|\bar{\nabla}f|^2=2\rho f$.  In this case, we assume the condition $tr_{\Sigma^{\perp}}\bar{\nabla}^2f\geq \rho k$ for some constant $k$ instead in Theorem \ref{th-4}.  Using the same argument as  in the proof of  Theorem \ref{th-4} by choosing $\alpha=\frac12, \beta=\rho, $  $ a=\rho(m-k)+\frac{-\inf_{\Sigma}(\bar{R}+\left|\mathbf{H}\right|^{2})-\inf_{\Sigma}(\bar{R}+|(\bar{\nabla}f)^{\perp}|^{2})+\sup_{\Sigma}\left|\mathbf{H}_{f}\right|^{2}}{2}  $, and  $\tilde{f}=f-\frac{\inf_{\Sigma}(\bar{R}+\left|(\bar{\nabla}f)^{\perp})\right|^{2})}{2\rho}$,  we get that $\Sigma$ has polynomial volume growth with the  rate $r^l$, where the  constant  $ l=(m-k)+\frac{-\inf_{\Sigma}(\bar{R}+\left|\mathbf{H}\right|^{2})-\inf_{\Sigma}(\bar{R}+|(\bar{\nabla}f)^{\perp}|^{2})+\sup_{\Sigma}\left|\mathbf{H}_{f}\right|^{2}}{2\rho}$.
\end{remark} 
\begin{remark} If $M$ is a complete shrinking gradient soliton with its Ricci curvatue tensor $\bar{Ric}$  bounded above,  $tr_{\Sigma^{\perp}}\bar{\nabla}^2f$ must be bounded below by some constant. 
 \end{remark} 
In \cite{CMZ2},  it was   proved that if a complete properly immersed submanifold $\Sigma$  in a complete Riemannian manifold  $M$ with $\bar{Ric}+\bar{\nabla}^2f\geq \frac12g$ and $|\bar{\nabla}f|^2\leq f$ has polynomial volume growth, it must have finite weighted volume (Proposition 5 in \cite{CMZ2}). On the other hand,   in \cite{AR}, Alencar-Rocha  proved that  if $\Sigma$ a complete submanifold immersed  in a complete Riemannian manifold with weighted mean curvature vector bounded in norm, then the finite weighted volume of $\Sigma$ implies properness of immersion (Proposition 1.1 in \cite{AR}).  Using these two properties and Theorem \ref{th-4}, we get Theorem \ref{th-equiv}.

\section{Lower estimate of volume growth}\label{section-lower}

In this section, we estimate  the lower bound of volume growth  of  the submanifolds with   weighted mean curvature vector bounded in norm. We assume that 
$(M,\bar{g},f)$ is a complete $m$-dimensional shrinking gradient Ricci soliton: 
\begin{align}\label{5-soliton}\bar{Ric}+\bar{\nabla}^2f=\dfrac12\bar{g}
\end{align}
with  the normalized potential function  $f$. We have that  $f$  satisfies (\ref{4-eq-soliton-2}), (\ref{4-eq-soliton-4}),   (\ref{4-eq-soliton-3}), and  (\ref{4-eq-soliton-5}). Moreover, the scalar curvature of $M$ satisfies $\bar{R}\geq 0$. 

Let  $n<m$ and $\Sigma$ be a complete $n$-dimensional properly immersed submanifold in $M$ with bounded $|{\bf H}_f|$.
We first prove that the  volume of $\Sigma$ is infinite. Before doing this,  we need some  facts.  

Observing that the conclusion of Lemma 4.1 in  \cite{W} is still true when  the assumption of boundedness of  the sectional curvature of the ambient manifold is changed to that  the sectional curvature  is only bounded above by a constant, we get the following the lemma.

\begin{lemma}\label{Vol-small}

Let $(M^m,\bar{g})$ be an $m$-dimensional Riemannian manifold whose sectional curvature is bounded above by $K_0>0$  and injectivity radius is bounded below by $i_0>0.$ Let $\Sigma$ be an $n$-dimensional  complete properly immersed submanifold in $(M^m,\bar{g})$. For any  $p\in \Sigma$ and  $r\leq \min\{1, i_0, 1/\sqrt{K_0}\}$, if $|\mathbf{H}|\leq  C/r$ in $B_r^M(p)\cap \Sigma$ for some positive constant $C>0$, then the  volume of $B_r^M(p)\cap\Sigma$  satisfies
\begin{align}\label{5-lem1-ineq}
\text{Vol}(B_r^M(p)\cap\Sigma)\geq \tau r^n,
\end{align}
where $\tau=\omega_ne^{-2(n\sqrt{K_0}+C)}$.  Here $\omega_n$ denotes the volume of the  unit ball in $\mathbb{R}^n$.
\end{lemma}
\begin{proof} Since $M$ has bounded geometry,  by the Hessian comparison theorem and direct computation, the following  inequality holds:
$$\bar{\nabla}^2\bar{r}(x)(V,V)-\frac1{\bar{r}(x)}|V-\left<V, \bar{\nabla}\bar{r}(x)\right>\bar{\nabla}\bar{r}(x)|^2\geq -\sqrt{K_0},$$
where  $x\in M$, $V$ denotes any unit tangential vector in $T_xM$, and $\bar{r}(x)=d_M(x, p)$ is the distance  between $p$ and $x$ in $M$.

Following  the same argument as in the proof of Lemma 4.1 in  \cite{W},  one can prove (\ref{5-lem1-ineq}). 

\end{proof}

To continue our proof, we need a logarithmic Sobolev inequality on submanifolds. In  \cite{E},   Ecker proved the logarithmic Sobolev inequality for submanifolds in $\mathbb{R}^m$ and pointed out that by applying the Sobolev inequality for submanifolds of Riemannian manifolds  (see \cite{HS}),  his theorem extends to this more general context restricted only by the additional
assumptions on the support of the admissible functions imposed there  and with constants depending apart from the dimension of the submanifold  also on the sectional curvatures of the ambient manifold.  For the sake of completeness, we state it here:
\begin{prop}\label{Sob} (see Proposition 3.1, \cite{W})  Let $\Sigma$ be an $n$-dimensional submanifold immersed  in  an $m$-dimensional  Riemannian manifold  $(M^m, \bar{g})$.  Assume that the  sectional curvature  of $M$ satisfies  $K_M \leq b^2$ and the injectivity radius   of $M$ restricted to $\Sigma$  is bounded below by a constant $i_0>0$, where $b$ denotes a positive real number or a pure imaginary one.    Let $\mu$ be a smooth positive function on $M$ and $\lambda>0$ be a positive constant. Then the following inequality
\begin{align}\label{ineq-soblev}
&\int_{\Sigma}h^2(\ln h)\mu d\sigma-\frac12\left(\int_{\Sigma}h^2\mu d\sigma\right)\ln\left(\int_{\Sigma}h^2\mu d\sigma\right)\nonumber\\
&\leq \lambda \int_{\Sigma}|\nabla h|^2\mu d\sigma +\frac{\lambda}{4}\int_{\Sigma}|\mathbf{H}-(\bar{\nabla}\ln\mu)^{\perp}|^2h^2\mu d\sigma\nonumber\\
&\quad+\int_{\Sigma}h^2\mu \left(c(n,\alpha,\lambda^{-1})-\frac{\lambda}{4}|\bar{\nabla}\ln \mu|^2-\frac{\lambda}{2}div(\bar{\nabla}\ln\mu)-\frac12\ln \mu\right)d\sigma
\end{align}
holds for any nonnegative $C^1$  function $h$ on  $\Sigma$  vanishing on $\partial \Sigma$ provided that
the volume of the support of $h$ (denoted by $|supp(h)| $) satisfies the following restriction
$$b^2(1 - \alpha)^{-\frac{2}{n}} (\omega_n^{-1}|supp(h)|)^{\frac2n} \leq 1\quad \text{and} \quad 2\rho_0 \leq i_0,$$
where 
\begin{equation*}
    \rho_0=\left\{
           \begin{array}{ll}
             b^{-1}\arcsin \left(b(1-\alpha)^{-\frac{1}{n}}(\omega_n^{-1}|supp(h)|)^{\frac1n}\right) & \hbox{if $b$ is real,} \\
             (1-\alpha)^{-\frac{1}{n}}(\omega_n^{-1}|supp(h)|)^{\frac1n} & \hbox{if $b$ is imaginary.}
           \end{array}
         \right.
\end{equation*}
Here $\alpha$ is a free parameter satisfying $0 < \alpha < 1 $ and  $\text{div}(\bar{\nabla}\ln\mu)$ denotes  the divergence of $\bar{\nabla}\ln\mu$  with respect to  $\Sigma$.

\end{prop}
 Now we go back to  the  gradient shrinking Ricci  soliton $(M,\bar{g},f)$. Take  $\mu=e^{-f}$. Then 
\begin{align}
\bar{\nabla}\ln\mu&=-\bar{\nabla}f,\nonumber\\
\mathbf{H}-(\bar{\nabla}\ln\mu)^{\perp}&=\mathbf{H}+(\bar{\nabla}f)^{\perp}=\mathbf{H}_f,\nonumber\\
div(\bar{\nabla}\ln\mu)&=-div\bar{\nabla}f=-\Delta f+\left<\mathbf{H}, (\bar{\nabla}f)^{\perp}\right>\nonumber\\
&=-\Delta f-|(\bar{\nabla}f)^{\perp}|^2+\left<\mathbf{H}_f, (\bar{\nabla}f)^{\perp}\right>.\nonumber
\end{align}
 Let  $\mu=e^{-f},$   $\alpha=\frac{n}{n+1},$ and $ \lambda=2$ in Proposition \ref{Sob}.  Using  the second equality in  (\ref{4-ineq-Delta}),   (\ref{4-eq-soliton-2}) and (\ref{4-eq-soliton-4}),  we have
\begin{align}\label{5-ineq-c}
&c(n,\alpha,\lambda^{-1})-\frac{\lambda}{4}|\bar{\nabla}\ln \mu|^2-\frac{\lambda}{2}div(\bar{\nabla}\ln\mu)-\frac12\ln \mu\nonumber\\
&=c(n)-\frac{1}{2}|\bar{\nabla}f|^2+\Delta f+|(\bar{\nabla}f)^{\perp}|^2-\left<\mathbf{H}_f, (\bar{\nabla}f)^{\perp}\right>+\frac{f}{2}\nonumber\\
&=c(n)-\frac{1}{2}|\bar{\nabla}f|^2+\bar{\Delta} f-tr_{\Sigma^{\perp}}\bar{\nabla}^2f+\frac{f}{2}\nonumber\\
&=c(n)+\dfrac12 m-\bar{R}-\frac12|\bar{\nabla}f|^2+\frac{f}{2}-tr_{\Sigma^{\perp}}\bar{\nabla}^2f\nonumber\\
&=c(n)+\dfrac12 m-\frac{1}{2}\bar{R}-tr_{\Sigma^{\perp}}\bar{\nabla}^2f.
\end{align}
Therefore  Proposition \ref{Sob} implies that

\begin{lemma}\label{lemma-soblev-H_f} Let $M$ and $\Sigma$ be the same  as in Proposition \ref{Sob}. Assume that   $(M^{m} , \overline{g}, f ) $ is a  gradient shrinking soliton.  Then
\begin{align}\label{ineq-soblev-2}
&\int_{\Sigma}h^2(\ln h)e^{-f} d\sigma-\frac12\int_{\Sigma}h^2e^{-f} d\sigma\ln\left(\int_{\Sigma}h^2e^{-f} d\sigma\right)\nonumber\\
&\leq 2 \int_{\Sigma}|\nabla h|^2e^{-f} d\sigma+\frac12\int_{\Sigma}|{\bf H}_f|^2h^2e^{-f}d\sigma\nonumber\\ 
&\qquad+\int_{\Sigma}h^2\left(c(n)+\dfrac12 m-\frac{1}{2}\bar{R}-tr_{\Sigma^{\perp}}\bar{\nabla}^2f\right) e^{-f}d\sigma,
\end{align}
holds for any nonnegative $C^1$  function $h$ on  $\Sigma$  vanishing on $\partial \Sigma$ provided that
the volume of the support of $h$ (denoted by $|supp(h)| $) satisfies the following restriction
\begin{align}
b^2 \left((n+1)\omega_n^{-1}|supp(h)|\right)^{\frac2n} \leq 1 \quad \text{and} \quad 2\rho_0 \leq i_0,
\end{align}\label{soblev-restriction}
where 
\begin{equation}\label{rho0}
    \rho_0=\left\{
           \begin{array}{ll}
             b^{-1}\arcsin \left[b\left((n+1)\omega_n^{-1}|supp(h)|\right)^{\frac1n}\right] & \hbox{if $b$ is real,} \\
             \left( (n+1)\omega_n^{-1}|supp(h)|\right)^{\frac1n} & \hbox{if $b$ is imaginary.}
           \end{array}
         \right.
\end{equation}

\end{lemma}

Now we are ready to prove the following result.

\begin{theorem}\label{Vol-infinite}
Let $\left(M,\bar{g},f\right)$ be an $m$-dimensional
 complete  shrinking  gradient Ricci soliton  whose sectional curvature is bounded above by a constant $K_0>0$  and injectivity radius is bounded below by a constant $i_0>0$. Let $\Sigma$ be an $n$-dimensional complete
properly immersed submanifold in $M$.  If $\Sigma$ has weighted
mean curvature vector bounded in norm, then its volume 
is infinite.
\end{theorem}

\begin{proof}
Suppose, to the contrary,  that $\Sigma$ has finite volume. Let $\rho(x)=2\sqrt{f(x)}$, $x\in M$. 
Fix a point    $p\in \Sigma$  and  denote $\bar{r}(x)=d_M(p,x)$. Since $\Sigma$ is non-compact and  properly immersed,   the image of $\bar{r}$ on $ \Sigma$ is   $[0,\infty)$. Note  (\ref{f-estimate}) says that there are constants  $r_0$ and $c$ so that $|\rho(x)-\bar{r}(x)|\leq c $ for $\bar{r}(x)\geq r_0, x\in M$. Thus  the image  of $\rho(x)$ on $ \Sigma\setminus B_{r_0}^M(p)\cap\Sigma$ contains $ (r_0+c, \infty)$. For $0< k_1< k_2$ very large, denote
\begin{align*}
{A}(2^{k_1},2^{k_2})&=\{x\in \Sigma; 2^{k_1}\leq \rho(x)\leq 2^{k_2}\}, \\
 {V}(2^{k_1},2^{k_2})&=\text{Vol}({A}(2^{k_1},2^{k_2})).
 \end{align*}
By  $Vol(\Sigma)<\infty$,  it holds  that for every $\epsilon>0$, there exists $k_0>0$ so that for $ k_2> k_1\geq  k_0$,
 \begin{align}\label{5-ineq-epsilon}
 {V}(2^{k_1},2^{k_2})\leq \epsilon.
 \end{align}
\begin{claim}\label{Claim-1} We may  choose some $k_1$ and $k_2$ in (\ref{5-ineq-epsilon}) so that $k_1$ and $k_2$ also satisfy the following inequality: 
\begin{align}\label{5-ineq-V}
{V}(2^{k_1},2^{k_2})&\leq 2^{4n}{V}(2^{k_1+2}, 2^{k_2-2}). 
\end{align}
\end{claim}
Now we show the claim. 
For a very large  $k$, the set  $$\{x\in \Sigma; 2^{k}+c\leq \bar{r}(x)\leq 2^{k+1}-c\}$$
contains  at least $2^{2k-1}$ disjoint balls in $M$ of radius $r=\frac{2^k-2c}{2^{2k}}$  centered at $x_i\in \Sigma$ and restricted on $\Sigma$:
 $$B_r(x_i)=(B_{r}^M(x_i))^{\text{o}}\cap \Sigma=\{x\in \Sigma; \bar{r}_{x_i}(x)=d_M(x_i,x)<r\}.$$
Since $\{x\in \Sigma; 2^{k}+c\leq \bar{r}(x)\leq 2^{k+1}-c\}\subseteq {A}(2^{k},2^{k+1})$,  the set  ${A}(2^{k},2^{k+1})$ also contains  all  the   balls $B_r(x_i)$, $i=1, \ldots, 2^{2k-1}.$

For  $ x\in A(2^k,2^{k+1})$, 
$$|\mathbf{H}|\leq |\mathbf{H}_f|+|(\bar{\nabla}f)^{\perp}|\leq \sup_{\Sigma}|\mathbf{H}_f|+\sqrt{f}\leq\sup_{\Sigma}|\mathbf{H}_f|+2^k\leq 2^{k+1}\leq \frac2r.$$
In the above inequality, we used Inequality (\ref{4-eq-soliton-5}) and $\sup_{\Sigma}|\mathbf{H}_f|\leq 2^k$ for $k$  very large.
Furthermore,  let $k$ be sufficiently large so that 
$$ 2^{-(k+1)}\leq r= 2^{-k}-2^{-2k+1}c\leq \min\{1,i_0,1/\sqrt{K_0}\}.$$ 
 Applying  Lemma \ref{Vol-small} to the  balls $B_r(x_i)$, we have
  $$\text{Vol}(B_r(x_i))\geq \tau r^n \geq \tau 2^{-(k+1)n},$$
   where  $\tau=\omega_ne^{-2(n\sqrt{K_0}+2)}.$  Hence 
\begin{align}\label{5-ineq-V-2}
{V}(2^k,2^{k+1}
)\geq \tau 2^{2k-1}2^{-(k+1)n}=\tau 2^{-kn+2k-n-1}.
\end{align}
Let $K$ be a very large integer satisfying $\tau 2^{8K-n-1}>\text{Vol}(\Sigma)$ and the conditions of $k$.  Take $k_1=2K, k_2= 6K+1.$ If (\ref{5-ineq-V}) in Claim \ref{Claim-1} fails, then 
$${V}(2^{k_1},2^{k_2})> 2^{4n}{V}(2^{k_1+2}, 2^{k_2-2}).$$
If $${V}(2^{k_1+2},2^{k_2-2})\leq 2^{4n}{V}(2^{k_1+4}, 2^{k_2-4}),$$
we are done, otherwise we  repeat the process. After $j$ steps we have 
\begin{align}\label{5-ineq-V-3}
{V}(2^{k_1},2^{k_2})> 2^{4nj}{V}(2^{k_1+2j}, 2^{k_2-2j}).
\end{align}
Take $j=K$. 
Then $k_1+2j=4K, k_2-2j=4K+1$. Thus (\ref{5-ineq-V-3}) and (\ref{5-ineq-V-2}) imply that 
\begin{align}\text{Vol}(\Sigma)\geq {V}(2^{k_1},2^{k_2}) \geq 2^{4nK}{V}(2^{4K},2^{4K+1})\geq \tau 2^{8K-n-1}.
\end{align}
This contradicts the assumption of $K$. Hence after finitely many steps,   (\ref{5-ineq-V}) must hold. Thus we proved Claim \ref{Claim-1}.

Now for each  $\epsilon>0$, let $k_1$ and $k_2$ satisfy (\ref{5-ineq-epsilon}) and (\ref{5-ineq-V}).
Take a  smooth cut-off function $\varphi(t)$  with $0\leq \varphi(t) \leq 1$,  $|\varphi'(t)|\leq 1$ and
\begin{align}
 \varphi(t)&=\left\{
           \begin{array}{lllll}
           0, & \hbox{ $t\leq 2^{k_1}$}\\
            0\leq \varphi'(t)\leq \frac{c_1}{2^{k_1}}, & \hbox{ $2^{k_1}\leq t\leq 2^{k_1+2}$}\\
            1, & \hbox{ $2^{k_1+2}\leq t\leq 2^{k_2-2}$} \\
 -\frac{c_2}{2^{k_2}} \leq \varphi'(t)\leq 0,&\hbox{ $2^{k_2-2}\leq t\leq 2^{k_2}$}\\
  0, & \hbox{ $t\geq 2^{k_2}$} 
           \end{array}
         \right.
\end{align}
where $c_1$ and $c_2$ are some positive constants. Define 
\begin{align}\label{h-def}
h(x)&=e^{L+\frac{f(x)}{2}}\varphi(\rho(x)), \quad x\in \Sigma,
\end{align}
where $L$ is a constant  and  $h(x)$ satisfies 
\begin{align}
1=\int_{\Sigma}h^2(x)e^{-f}=e^{2L}\int_{{A}(2^{k_1}, 2^{k_2})}\varphi^2(\rho(x)) .
\end{align}
Choose a very small  $\epsilon_0>0$  satisfying  
\begin{align}
&\sqrt{K_0}\left((n+1)\omega_n^{-1}\epsilon_0\right)^{\frac1n}\leq 1, \label{epsilon-01}\\
 \text{and}&\quad \frac{2}{\sqrt{K_0}}\arcsin\left[\sqrt{K_0}\left((n+1)\omega_n^{-1}\epsilon_0\right)^{\frac1n}\right]\leq i_0.\label{epsilon-02}
\end{align}
Note that for each $\epsilon\leq \epsilon_0$, by the definition of $h$ and (\ref{5-ineq-epsilon}), it holds that
$$|supp(h)|\leq V(2^{k_1}, 2^{k_2})\leq \epsilon\leq \epsilon_0.$$
Hence (\ref{epsilon-01}) and (\ref{epsilon-02}) imply that $h(x)$ satisfies the restriction of $|supp(h)|$ provided in Lemma \ref{lemma-soblev-H_f}. Substituting $h(x)$ in (\ref{h-def})  into the logarithmic Sobolev inequality (\ref{ineq-soblev-2}), we have
\begin{align}\label{5-sob-ineq-apply-1}
&\int_{A(2^{k_1}, 2^{k_2})}e^{2L}\varphi^2(\rho(x))\bigg(L+\frac{f}{2}+\ln(\varphi(\rho(x))\bigg)\nonumber\\
&\leq \bar{C}+2\int_{A(2^{k_1}, 2^{k_2})}e^{2L}|\varphi'(\rho(x))\nabla\rho+
\frac12\varphi(\rho(x))\nabla f|^2,
\end{align}
where $\bar{C}=\frac12\sup_{\Sigma}|{\bf H}_f|^2+c(n)+\frac m2-\inf_{\Sigma}\bar{R}-\inf_{\Sigma} tr_{\Sigma^{\perp}}\bar{\nabla}^2f$.  Here we used the fact that $tr_{\Sigma^{\perp}}\bar{\nabla}^2f$ must be bounded below by some constant. The reason is the following: Since $\bar{M}$ has the sectional curvature bounded above by a positive constant, its Ricci curvatue tensor $\bar{Ric}$ is bounded above.  By the equation $\bar{Ric}+\bar{\nabla}^2f=\frac12\bar{g}$, we have that $\bar{\nabla}^2f$ is bounded below by some constant $2$-tensor. Now (\ref{5-sob-ineq-apply-1}) implies that
\begin{align}\label{5-sob-apply}
\bar{C}&\geq L+\int_{A(2^{k_1}, 2^{k_2})}e^{2L}\varphi^2(\rho(x))\bigg(\frac{f-|\nabla f|^2}{2}\bigg)\nonumber\\
&\quad +\int_{A(2^{k_1}, 2^{k_2})}e^{2L}\varphi^2(\rho(x))\ln(\varphi(\rho(x))\nonumber\\
&\quad-2\int_{A(2^{k_1}, 2^{k_2})}e^{2L}|\varphi'(\rho(x))\nabla\rho|^2\nonumber\\
&\quad-2 \int_{A(2^{k_1}, 2^{k_2})}e^{2L}\varphi'(\rho(x))\varphi(\rho(x))\langle\nabla\rho, \nabla f\rangle .
\end{align}
In the following, we estimate the integrals in (\ref{5-sob-apply}). 

By  $|\nabla f|^2\leq f,$  the first integral is non-negative. 

Since  the inequality $t\ln t\geq -\frac1{e} $ holds  for any $0\leq t\leq 1$,  
$$\int_{A(2^{k_1}, 2^{k_2})}e^{2L}\varphi^2(\rho(x))\ln(\varphi(\rho(x))\geq -\frac1{2e}e^{2L}V(2^{k_1}, 2^{k_2}).$$
  By $|\varphi'(t)|\leq 1$ and  $|\nabla \rho|\leq 1$, 
  $$-2\int_{A(2^{k_1}, 2^{k_2})}e^{2L}|\varphi'(\rho(x))\nabla\rho|^2\geq -2e^{2L}V(2^{k_1},2^{k_2}).$$
Note that $0\leq \langle\nabla \rho, \nabla f\rangle=\frac{|\nabla f|^2}{\sqrt{f}}\leq \sqrt{f}=\frac{\rho}{2}$ and the property that $0\leq \varphi'\leq \frac{c_1}{2^{k}}$ in  $[2^{k_1}, 2^{k_1+2}]$;  $\varphi'\leq 0$ in $[2^{k_1+2}, 2^{k_2}]$. Then
\begin{align*}
-2 \int_{A(2^{k_1}, 2^{k_2})}e^{2L}\varphi'(\rho(x))\varphi(\rho(x))\langle\nabla\rho, \nabla f\rangle
&\geq -2e^{2L}\int_{A(2^{k_1},2^{k_1+2})}\frac{c_1}{2^{k_1}}\frac{2^{k_1+2}}{2}\\
&\geq -4c_1e^{2L}V(2^{k_1}, 2^{k_2}).
\end{align*} 
Thus
(\ref{5-sob-apply})  implies  that
\begin{align}
\bar{C}
&\geq L-(\frac{1}{2e}+2+4c_1)e^{2L}V(2^{k_1}, 2^{k_2}) .
\end{align}
Using (\ref{5-ineq-V}),  we have
\begin{align}\label{5-ineq-L}
\bar{C}
&\geq L-(\frac{1}{2e}+2+4c_1)e^{2L}2^{4n}V(2^{k_1+2}, 2^{k_2-2})\nonumber\\
&\geq L-(\frac{1}{2e}+2+4c_1)2^{4n}\int_{A(2^{k_1}, 2^{k_2})}e^{2L}\varphi^2(\rho(x))\nonumber\\
&\geq L-(\frac{1}{2e}+2+4c_1)2^{4n} .
\end{align}
So $L$ is bounded above by a universal constant.
On the other hand,  using (\ref{5-ineq-epsilon}), we have
\begin{align}
1=\int_{A(2^{k_1}, 2^{k_2})}e^{2L}\varphi(\rho(x))\leq e^{2L}V(2^{k_1}, 2^{k_2})\leq e^{2L}\epsilon .
\end{align}
Since $\epsilon$ can be arbitrarily small,  $L$ cannot be bounded above by a universal constant, which is a contradiction. This contradiction implies that the volume of $\Sigma$ is infinite. So we finished the proof.

\end{proof}

Before proving Theorem  \ref{th-3}, we show   some inequalities about the volume of  $B^M_{r}(p)\cap \Sigma, $ where $p\in \Sigma$.  Suppose that  $tr_{{\Sigma}^{\perp}}\bar{\nabla}^2f\geq \frac k2$ for some constant $k$. So
(\ref{4-ineq-Delta}) implies
\begin{align}\label{ineq-5.1}\Delta f+\frac12\bar{R}+\dfrac12|(\bar{\nabla}f
)^{\perp}|^{2} \leq  \dfrac12s,
\end{align}
 where $s=m-k-\inf_{\Sigma}(\overline{R}+|\mathbf{H}|^2)+\sup_{\Sigma}|\mathbf{H}_f|^2$ is a nonegative constant.  
Denote  the set
$D_r=\{x\in \Sigma:\rho(x)\leq r\}$, for $r\geq r_0$, where $r_0$ is chosen to let $D_r$ have the positive measure. We define
\begin{align*}
V(r)&=\text{Vol}(D_r)=\int_{D_r}dv_{\Sigma},\\
  \chi(r)&=\int_{D_r}\bar{R}dv_{\Sigma},
\quad  \text{and} \quad \eta(r)=\int_{D_r}|\left(\bar{\nabla}f\right)^{\perp}|^{2}dv_{\Sigma}.\nonumber
\end{align*}
By the co-area formula, it holds that
\begin{align*}
V'(r)&=\int_{\partial D_r}\frac{1}{|\nabla \rho|}dA,\\
 \chi'(r)&=\int_{\partial D_r}\frac{\bar{R}}{|\nabla \rho|}dA, \quad\text{and}\quad  \eta'(r)=\int_{\partial D_r}\frac{|\left(\bar{\nabla}f\right)^{\perp}|^{2}}{|\nabla \rho|}dA.
\end{align*}
We will prove the following inequality:
\begin{prop}\label{5-prop-1} Under the above notation and assumption, for $r\geq r_0$, it holds that
\begin{align}\label{5-eta-1}rV'(r)-sV(r)\leq \frac{4}{r}\eta'(r)-\eta(r)+\frac 4r\chi'(r)-\chi(r).
\end{align}

\end{prop}
\begin{proof}


 By Inequality (\ref{ineq-5.1}), the Stokes' formula, and (\ref{4-eq-soliton-4}),  we have
\begin{align*}
\dfrac{s}{2}V(r)-\dfrac12\eta(r)-\frac12\chi(r)&\geq \int_{D_r}\Delta fdv_{\Sigma}=\int_{\partial D_r}\left<\nabla f, \frac{\nabla \rho}{|\nabla \rho|}\right>\\
&=\int_{\partial D_r}\left<\nabla f, \frac{\nabla f}{|\nabla\rho|\sqrt{f}}\right>=\int_{\partial D_r}\dfrac{|\overline{\nabla}f|^2-|(\overline{\nabla}f)^{\perp}|^2}{|\nabla\rho|\sqrt{f}}\\
&=\dfrac2r\int_{\partial D_r}\dfrac{f-\bar{R}-|(\overline{\nabla}f)^{\perp}|^2}{|\nabla \rho|}\\
&=\frac r2V'(r)-\frac2r\eta'(r)-\frac2r\chi'(r).
\end{align*}

\end{proof}

\begin{remark} From the  above proof, we know that   for all $r\geq r_0$,
\begin{align}\label{5-ineq-eta}
\eta(r)+\chi(r)\leq sV(r).
\end{align}
\end{remark}
Proposition \ref{5-prop-1} implies that
\begin{lemma} \label{5-prop-2} Under the same notation and assumption as in Proposition \ref{5-prop-1},  for any $r_2\geq r_1\geq \max\{r_0, 2\sqrt{s+2}\}$,
\begin{align}\label{5-ineq-2}
r_2^{-s}V(r_2)-r_1^{-s}V(r_1)
\leq 4sr_2^{-s-2}V(r_2).
\end{align}
\end{lemma}
\begin{proof} Using (\ref{5-eta-1}), we have
\begin{align}\label{5-prop-2-1} 
(r^{-s}V(r))'&=r^{-s-1}\left[rV'(r)-sV(r)\right]\nonumber\\
&\leq r^{-s-2}\left[4\eta'(r)-r\eta(r)+4\chi'(r)-r\chi(r)\right] .
\end{align}
Integrate (\ref{5-prop-2-1})  from $r_1$ to $r_2$. Then 
\begin{align}
 &r_2^{-s}V(r_2)- r_1^{-s}V(r_1)\nonumber\\
 &\leq  \int_{r_1}^{r_2}4r^{-s-2}\left[\eta'(r)+\chi'(r)\right]dr-\int_{r_1}^{r_2}r^{-s-1}\left[\eta(r)+\chi(r)\right]dr\nonumber\\
&=4r_2^{-s-2}\left[\eta(r_2)+\chi(r_2)\right]-4r_1^{-s-2}\left[\eta(r_1)+\chi(r_1)\right]\nonumber\\
&\qquad +\int_{r_1}^{r_2}\left[4(s+2)r^{-s-3}-r^{-s-1}\right]\left[\eta(r)+\chi(r)\right]dr.\nonumber
\end{align}
Since $r_2\geq r_1\geq 2\sqrt{s+2}$,
\begin{align}
\label{5-ineq-volume}
 r_2^{-s}V(r_2)- r_1^{-s}V(r_1)&\leq 4r_2^{-s-2}\left[\eta(r_2)+\chi(r_2)\right]-4r_1^{-s-2}\left[\eta(r_1)+\chi(r_1)\right]\nonumber\\
&\qquad +\left[\eta(r_1)+\chi(r_1)\right]\int_{r_1}^{r_2}\left[4(s+2)-r^2\right]r^{-s-3}dr\nonumber\\
&\leq   4r_2^{-s-2}\left[\eta(r_2)+\chi(r_2)\right]-4r_1^{-s-2}\left[\eta(r_1)+\chi(r_1)\right]\nonumber\\
&\qquad +\left[\eta(r_1)+\chi(r_1)\right]\left[-4(r_2^{-s-2}-r_1^{-s-2})\right]\nonumber\\
&=4r_2^{-s-2}\left[\eta(r_2)+\chi(r_2)-\eta(r_1)-\chi(r_1)\right]\nonumber\\
&\leq 4sr_2^{-s-2}V(r_2).
\end{align}
In the last inequality in (\ref{5-ineq-volume}), we used (\ref{5-ineq-eta}).

\end{proof}
Lemma \ref{5-prop-2} also gives an alternative proof of polynomial growth of a properly immersed submanifold in a shrinking gradient Ricci soliton with bounded $|{\bf H}_f|$.  More precisely,
\begin{cor}\label{5-cor-poly}
 Let $\Sigma$ be a properly immersed complete $n$-dimensional submanifold in an $m$-dimensional  
  complete shrinking gradient Ricci soliton $(M^{m} , \overline{g}, f ) $ with  $tr_{{\Sigma}^{\perp}}\bar{\nabla}^2f\geq \frac k2$ for some constant $k$. Then,  for $p\in M$ fixed, there are some constants $C>0$ and $r_0$ so that 
\begin{align}\label{cor-volume}\text{Vol}(B^M_r(p)\cap \Sigma)\leq Cr^s  \quad \text{for}\quad r\geq r_0,
\end{align}
where  $s=m-k-\inf_{\Sigma}(\overline{R}+|\mathbf{H}|^2)+\sup_{\Sigma}|\mathbf{H}_f|^2$. 
\end{cor}

\begin{proof}  In Proposition \ref{5-prop-2}, letting  $r_1$  fixed and taking  $r_2=r$,  we have
$$r^{-s}V(r)-r_1^{-s}V(r_1)
\leq 4sr^{-s-2}V(r).$$
Then $$V(r)\leq \frac{1}{1-4sr^{-2}}\frac{V(r_1)}{r_1^s}r^s\leq \frac{V(r_1)}{(1-4sr_1^{-2})r_1^s}r^s=\bar{C}r^s.$$
Now using the same argument as  in the proof of Theorem \ref{th-4} (iii), we get (\ref{cor-volume}).

\end{proof}
\begin{remark}  It is worth noting that   the polynomial volume growth estimate in Theorem \ref{th-4} is better than  the one  in Corollary \ref{5-cor-poly} in the sense that $l\leq s$.
\end{remark}
To prove Theorem \ref{th-3}, we also need the following logarithmic Sobolev inequality by taking $\mu=1$, $\alpha=\frac{n}{n+1}$ and $\lambda=1$ in Proposition \ref{Sob}.
\begin{lemma}\label{Sob-1-lemma} Let $M$ and $\Sigma$ be the same as in Proposition \ref{Sob}. Then
\begin{align}
&\int_{\Sigma}h^2(\ln h) d\sigma-\frac12\left(\int_{\Sigma}h^2 d\sigma\right)\ln(\int_{\Sigma}h^2 d\sigma)\nonumber\\
&\leq  \int_{\Sigma}|\nabla h|^2 d\sigma +\frac14\int_{\Sigma}|\mathbf{H}|^2h^2d\sigma+c(n)\int_{\Sigma}h^2 d\sigma, \label{5-ineq-sob-linear}
\end{align}
where the function $h$ is the same as in Proposition \ref{Sob} and  $c(n)$ is the constant  $c(n,\frac{n}{n+1},1)$ in Proposition \ref{Sob}.
\end{lemma}

Now we are ready to prove Theorem \ref{th-3}. 

\bigskip
\noindent{\it Proof of  Theorem \ref{th-3}}.
  First, note that the property of bounded geometry  of $M$ and the soliton equation (\ref{5-soliton}) of $M$ imply that  $tr_{{\Sigma}^{\perp}}\bar{\nabla}^2f\geq \frac k2$ for some constant $k$. Now  we give the following
\begin{claim} \label{claim-2} There are constants $C>0$ and $ \tilde{r}>0$ so that 
\begin{align}
 V(r)\geq Cr \quad \text{for all} \quad r\geq \tilde{r},
\end{align}
where $V(r)=\text{Vol}(D_r)$.
\end{claim}
We will prove the claim by contradiction. Assume that for any $\epsilon >0$ and any $\tilde{r}>0$, there exists a number $r\geq\tilde{r}$ depending on $\epsilon$ and $\tilde{r}$ so that 
\begin{align}\label{5-ineq-1}
V(r)\leq \epsilon r.
\end{align}
 For $t\geq \max\{r_0, 2\sqrt{s+2}\}$,  Lemma \ref{5-prop-2} implies 
\begin{align}
V(t+1)&\leq V(t)\frac{(t+1)^{s}}{t^{s}}\frac{1}{1-\frac{4s}{(t+1)^2}}.\nonumber
\end{align}
Using the basic inequality $\frac1{1-a}\leq 1+2a$ for  $0<a\leq\frac12$,
we have
\begin{align*}
V(t+1)
&\leq V(t)\left(1+\frac1t\right)^{s}\left(1+\frac{8s}{(t+1)^2}\right).
\end{align*}
Using $(1+\frac1t)^s\leq 2^s$, we get 
\begin{align*}
V(t+1)-V(t)&\leq V(t)\left((1+\frac1t)^{s}+\frac{C_1(s)}{t^2}-1\right),
\end{align*}
where $C_1(s)$  is a constant depending on $s$. So
\begin{align}\label{5-ineq-c-1}
V(t+1)-V(t)&\leq V(t)\left((1+\frac1t)^{[s]+1}+\frac{C_1(s)}{t^2}-1\right)\nonumber\\
&\leq V(t)\frac{C_2(s)}{t},
\end{align}
where $C_2(s)$ is a  constant depending on $s$. The last inequality in (\ref{5-ineq-c-1}) used Bernoulli inequality: $(1+z)^{\alpha}\leq 1+(2^{\alpha}-1)z $ for $z\in [0,1]$ and $\alpha\geq 1$. By (\ref{5-ineq-c-1}),  there exists a $\bar{r}\geq \max\{r_0,2\sqrt{s+2}\}$ depending on $s$, such that, for  $t\geq \bar{r}$,  
\begin{align}\label{5-ineq-5}
V(t+1)\leq 2V(t).
\end{align}
Hence, for $t\geq \bar{r}+1$,   (\ref{5-ineq-c-1})  and (\ref{5-ineq-5}) imply  that
\begin{align}\label{5-ineq-c-2}
V(t+2)-V(t-1)&\leq C_2(s)\left(\frac{V(t+1)}{t+1}+\frac{V(t)}{t}+\frac{V(t-1)}{t-1}\right)\nonumber\\
&\leq C_2(s)V(t)\left(\frac{2}{t+1}+\frac1t+\frac1{t-1}\right)\nonumber\\
&\leq C(s)\frac{V(t)}{t},
\end{align}
where $C(s)=5C_2(s)$. In the last inequality of (\ref{5-ineq-c-2}), we used $\frac1{t-1}\leq \frac2t$ for $t\geq 2$.

We may further assume that for an $\epsilon>0$ to be determined later,  there is an integer   $r\geq \bar{r}+1$ so that
\begin{align}\label{5-vol-epsilon}
V(r)\leq \epsilon r.
\end{align}
Otherwise,  (\ref{5-ineq-1}) fails and hence Claim \ref{claim-2} is proved. For $\epsilon$ and $r$ in (\ref{5-vol-epsilon}), define the set $S$:
$$S:=\{k\in \mathbb{N}: V(t)\leq 2\epsilon r\quad  \text{for all integers } t\in [r,k]\}.$$
Clearly $r\in S$ and $S\neq \emptyset$. By (\ref{5-ineq-c-2}) and the definition of $S$, for $k\in S$ and $ t\in [r,k]$,
\begin{align}\label{5-ineq-small}
V(t+2)-V(t-1)\leq C(s)\frac{2\epsilon r}{t}\leq 2\epsilon C(s).
\end{align}
\begin{claim}\label{claim-3} Any integer $k\geq r$ is in $S$. 
\end{claim}

We will prove Claim \ref{claim-3} by induction. Assume $k\in S$.  Let  $h(x)$ be the Lipschitz function on $\Sigma$ with compact support defined by
\begin{align}\label{5-def-h}
 h(x)&=\left\{
           \begin{array}{llll}
           0, & \hbox{in $D(t-1)$} \\
           \rho(x)-(t-1),  &\hbox{in $D(t)\backslash D(t-1)$}\\
            1, & \hbox{in $D(t+1)\backslash D(t)$} \\
             t+2-\rho(x), & \hbox{in $D(t+2)\backslash D(t+1)$}\\
0, &\hbox{in $\Sigma\backslash D(t+2)$.}
           \end{array}
         \right.
\end{align}
 Since $supp(h)\subseteq D(t+2)\backslash D(t-1)$,  (\ref{5-ineq-small}) implies $|supp(h)|<2\epsilon C(s)$. 
We may choose $\epsilon$ small enough such that $2\epsilon C(s)<\epsilon_0$, where $\epsilon_0$ is the constant chosen by (\ref{epsilon-01}) and (\ref{epsilon-02}). So $|supp(h)|<\epsilon_0$.  Substitute $h(x)$ into  the  Log-Sobolev inequality in Lemma \ref{Sob-1-lemma}.
 Then
\begin{align*}
&\int_{D(t+2)\backslash D(t-1)}h^2(\ln h) d\sigma-\frac12\left(\int_{\Sigma}h^2 d\sigma\right)\ln \left( V(t+2)-V(t-1)\right)\nonumber\\
&\leq \int_{D(t+2)\backslash D(t-1)}|\nabla\rho|^2d\sigma+c(n)[ V(t+2)-V(t-1)]+\frac14\int_{\Sigma}|\mathbf{H}|^2h^2d\sigma .
\end{align*}
Noting $|\nabla\rho|\leq 1$, the basic inequality $s\ln s\geq -\frac1e,$ for $ 0\leq s\leq 1$,  the inequality  $|\mathbf{H}|^2\leq 2(|\mathbf{H}_f|^2+|\mathbf{H}-\mathbf{H}_f|^2)$, and the boundedness of $|\mathbf{H}_f|^2$, we have
\begin{align}\label{5-ineq-c-3-3}
&-\frac12\left(\int_{\Sigma}h^2 d\sigma\right)\ln \left(V(t+2)-V(t-1)\right)\nonumber\\
&\leq \bar{C}[ V(t+2)-V(t-1)]+\frac12\int_{\Sigma}|\mathbf{H}-\mathbf{H}_f|^2h^2\nonumber\\
&\leq\bar{C}[V(t+2)-V(t-1)]+\frac12[\eta(t+2)-\eta(t-1)],
\end{align}
where $\bar{C}=c(n)+1+\frac{1}{2e}+\frac12\sup_{\Sigma}|\mathbf{H}_f|^2$. We may let $\epsilon$ be very small  so that $2\epsilon C(s)\leq 1$.  By  (\ref{5-ineq-small}), 
$$-\ln \left(V(t+2)-V(t-1)\right)\geq  -\ln  (2\epsilon C(s))\geq 0.$$
Noting  $\int_{\Sigma}h^2\geq V(t+1)-V(t)$ and using the above inequality, (\ref{5-ineq-c-3-3}) implies
\begin{align}\label{5-ineq-c-3}
& [V(t+1)-V(t)]\ln (2\epsilon C(s))^{-1}\nonumber\\
&\leq 2\bar{C}[V(t+2)-V(t-1)]+\eta(t+2)-\eta(t-1).
\end{align}
Iterating  (\ref{5-ineq-c-3}) from $t=r$ to $t=k$ and summing up give that
\begin{align}\label{5-ineq-c-4}
&[V(k+1)-V(r)]\ln(2\epsilon C(s))^{-1}\nonumber\\
&\leq 2\bar{C}[V(k+2)+V(k+1)+V(k)-V(r+1)-V(r)-V(r-1)]\nonumber\\
&\quad +\eta(k+2)+\eta(k+1)+\eta(k)-\eta(r+1)-\eta(r)-\eta(r-1)\nonumber\\
&\leq 6\bar{C}V(k+2)+3\eta(k+2)\nonumber\\
&\leq \left(6\bar{C}+3s\right)V(k+2)\nonumber\\
&\leq \left(12\bar{C}+6s\right)V(k+1).
\end{align}
Here we used   (\ref{5-ineq-eta})  and (\ref{5-ineq-5}) in the third and last inequalities in (\ref{5-ineq-c-4}) respectively.   Then  (\ref{5-ineq-c-4}) implies that
\begin{align}\label{5-ineq-c-5}
V(k+1)&\leq V(r)\frac{\ln(2\epsilon C(s))^{-1}}{\ln(2\epsilon C(s)) ^{-1}-12\bar{C}-6s}.
\end{align}
Further choose $\epsilon$ very small so  that
$$\frac{\ln(2\epsilon C(s))^{-1}}{\ln(2\epsilon C(s))^{-1}-12\bar{C}-6s}\leq 2.$$
Using  the assumption (\ref{5-vol-epsilon}), (\ref{5-ineq-c-5}) reduces to
\begin{align}
V(k+1)&\leq 2V(r)\leq  2\epsilon r.
\end{align}
So $k+1\in S$. By induction,  Claim \ref{claim-3} holds. Now, for any integer $k\geq r$,
$$V(k)\leq 2\epsilon r.$$
Hence $\Sigma$ must have finite volume, which contradicts Theorem \ref{th-3}.  By this contradiction, we have proved Claim \ref{claim-2}.

Now given a fixed point $p\in M$, by  (\ref{f-estimate}),  there are constants $c$ and $r_1$ so that $|\rho(x)-\bar{r}(x)|<c $ for $\bar{r}(x)\geq r_1,$ where $ x\in \Sigma$. 
Then for $r\geq r_1$,
\begin{align}\label{5-subset}
 D_r\subseteq B^M_{r+c}(p)\cap \Sigma.
 \end{align}
By (\ref{5-subset}) and Claim \ref{claim-2}, for $r\geq \max\{\tilde{r}+c, r_1+c, 2c\}$,
$$\text{Vol}(B^M_{r}(p)\cap \Sigma)\geq \text{Vol}(D_{r-c})\geq C(r-c)\geq \frac{C}{2}r.$$
Theorem \ref{th-3} is proved.

\qed

\end{document}